\newfont{\Bbb}{msbm10 scaled\magstephalf}
\documentclass[reqno]{amsart}
\usepackage{amsfonts}
\usepackage{amsmath, amsthm, amscd, amsfonts}
\usepackage{amssymb}
\newtheorem{thm}{Theorem}[section]
 \newtheorem{cor}[thm]{Corollary}
 \newtheorem{lem}[thm]{Lemma}
 
 \theoremstyle{definition}

\theoremstyle{remark}
 \newtheorem{rem}[thm]{Remark}
 
 \numberwithin{equation}{section}

\begin{document}

\title[Hypercyclic Composition Operators]{Hypercyclic composition operators on the little Bloch space $\mathcal{B}_0$ and the Besov spaces $B_p\;$}

\author[Y.X. Laing and Z.H.Zhou]{Yu-Xia Liang and Ze-Hua Zhou$^{*}$}

\address{\newline Yu-Xia Liang\newline School of Mathematical Sciences,
Tianjin Normal University, Tianjin 300387, P.R. China.}
\email{liangyx1986@126.com}

\address{\newline  Ze-Hua Zhou\newline Department of Mathematics, Tianjin University, Tianjin
300072, P.R. China.\newline
Center for Applied Mathematics, Tianjin University, Tianjin 300072,
P.R. China.} \email{zehuazhoumath@aliyun.com;
zhzhou@tju.edu.cn}

\keywords{Hypercyclic, composition operator,  Bloch space, Besov space}
\subjclass[2010]{Primary: 47A16; Secondary: 47B38, 47B37.}%
\date{}

\date{}
\thanks{\noindent $^*$Corresponding author.\\
The authors were supported in part by the National Natural Science Foundation of China (Grant Nos. 11371276), the Doctor Fund of Tianjin Normal University (Grant Nos. 52XB1514).}

\begin{abstract}
  Let $S(\mathbb{D})$ be the collection of all  holomorphic self-maps on $\mathbb{D}$ of the complex plane $\mathbb{C}$, and $C_{\varphi}$  the composition operator  induced by $\varphi\in S(\mathbb{D})$. We obtain that there are no hypercyclic composition operators on the little Bloch space $\mathcal{B}_0$ and the Besov space $B_p$.
\end{abstract}







\maketitle

\section{Introduction}

Let $\mathbb{D}=\{z\in \mathbb{C}:\;|z|<1\}$ be the unit disc in the complex plane $\mathbb{C}$ and $S(\mathbb{D})$ be the collection of all   holomorphic self-maps on $\mathbb{D}$. We denote $dA(z)=dx dy$  the Lebesgue area measure on $\mathbb{C}$. For the composition operator $C_{\varphi}$  induced by $\varphi\in S(\mathbb{D})$ is defined as $$C_\varphi f(z)=f\circ \varphi(z) ,\;f\in H(\mathbb{D}),\;z\in \mathbb{D}.$$

The one-to-one holomorphic functions that map $\mathbb{D}$ onto itself, called the \emph{M\"{o}bius} transformations, and denoted by $\mathcal{M}$ (also $Aut(\mathbb{D})$), have the form $\lambda \varphi_a,$ where $|\lambda|=1$ and $\varphi_a$ is the basic conformal automorphism defined by
$$\varphi_a(z)=\frac{a-z}{1-\bar{a}z},\;\;z\in \mathbb{D},$$  for $a\in \mathbb{D}$. The following identities are easily verified:
\begin{eqnarray}&&1-|\varphi_a(z)|^2=\frac{(1-|a|^2)(1-|z|^2)}{|1-\bar{a}z|^2}\nonumber\\\mbox{and}\nonumber
\\&&(1-|z|^2)|\varphi_a'(z)|=1-|\varphi_a(z)|^2.\label{1.1}\end{eqnarray}

 A linear space $X$ of analytic functions on the open unit disk $\mathbb{D}$ is said to be
\emph{M\"{o}bius-invariant}, \;if $f\circ S \in X$ for all $f\in X$ and all  $S\in \mathcal{M}$  and $X$ has
a seminorm $\|\;\|_X$ such that $\|f\circ S\|_X=\|f\|_X$ for each $f \in X$ and each  $S\in \mathcal{M}$.

The well-known\emph{ M\"{o}bius-invariant} function space--the Besov spaces $B_p$ $(1<p<\infty)$  are defined as following
$$B_p=\{f\in H(\mathbb{D}):\;\;\|f\|_{B_p}^p=\int_{\mathbb{D}} |f'(z)|^p (1-|z|^2)^{p-2} dA(z)<\infty\},$$ that is, $f\in B_p$ if and only if the function $(1-|z|^2)f'\in L^p(\mathbb{D},d\lambda),$ where
$$d\lambda(z)=\frac{dA(z)}{(1-|z|^2)^2}.$$  Although the measure $\lambda$ is not a finite measure on $\mathbb{D},$ it is a \emph{M\"{o}bius-invariant}. Indeed, by (\ref{1.1}) $$d\lambda (\varphi_a(z))=\frac{|\varphi_a'(z)|^2}{(1-|\varphi_a(z)|^2)^2} dA(z)=\frac{dA(z)}{(1-|z|^2)^2}=d\lambda(z).$$
Hence we have the following change-of-variable formula
\begin{eqnarray*}\int_{\mathbb{D}} f\circ \varphi_a(z)  d\lambda(z)=\int_{\mathbb{D}} f(u) d\lambda(u), \end{eqnarray*} for every positive measurable function $f$ on $\mathbb{D}$. From which it is easily seen that \begin{eqnarray}\|f\circ \varphi_a\|_{B_p}=\|f\|_{B_p},\label{1.2}\end{eqnarray} and the above identity  also holds for $S=\lambda \varphi_a \in \mathcal{M}$ with $|\lambda|=1$. Thus  \begin{eqnarray*}\;\mbox{if}\;f\in B_p\;\;\mbox{then}\;\; f\circ S\in B_p\;\;\mbox{for all}\;\;S\in \mathcal{M}. \end{eqnarray*} That is, $B_p\;(1<p<\infty)$ are \emph{M\"{o}bius-invariant} spaces.

For $p=1,$ the  Besov space $B_1$  consists of the analytic functions $f$ on $\mathbb{D}$ that admit the representation $$f(z)=\sum_{n=1}^{\infty} a_n \varphi_{\lambda_n}(z),\;\;z\in \mathbb{D},$$
where $\{a_n\}\in l^1$ and $\lambda_n\in \mathbb{D}$ for $n\in \mathbb{N}.$ The norm in $B_1$ is defined as
$$\|f\|_{B_1}=\inf \left\{ \sum_{n=1}^{\infty}|a_n|:\;f(z)=\sum_{n=1}^{\infty} a_n \varphi_{\lambda_n}(z),z\in \mathbb{D}\right\}.$$
It is evident that $B_1$ is the \emph{M\"{o}bius} invariant subset of the bounded analytic functions space $H^\infty$. On the other hand, $B_1$ has the following definition,
$$B_1=\{f\in H(\mathbb{D}):\;\;\|f\|_{B_1}=\int_{\mathbb{D}} |f''(z)|dA(z)<\infty\},$$ even though the above semi-norm is not \emph{M\"{o}bius-invariant}, the Besov space $B_1$ is the minimal \emph{M\"{o}bius-invariant} space (see, e.g. \cite{AF1, AF2}).

 It is well-known that   $B_p\; (1< p<\infty)$ are Banach spaces endowed with the norm denoted by $\|f\|_p$, $$\|f\|_p^p=|f(0)|^p+\|f\|_{B_p}^p.$$ Another \emph{M\"{o}bius-invariant} space of analytic functions on $\mathbb{D}$ is the Bloch space $\mathcal{B}$, $$\mathcal{B}=\{f\in H(\mathbb{D}):\;\|f\|_{\mathcal{B}}=\sup\limits_{z\in \mathbb{D}}(1-|z|^2)|f'(z)|<\infty\}.$$
By (\ref{1.1}) it follows that \begin{eqnarray}\|f\circ\varphi_{a}\|_\mathcal{B}&=&\sup\limits_{z\in \mathbb{D}} (1-|z|^2)|(f\circ\varphi_a)'(z)|\nonumber\\&=&\sup\limits_{z\in \mathbb{D}}(1-|z|^2)|f'(\varphi_a(z))\varphi_a'(z)|\nonumber\\&=&\sup\limits_{z\in \mathbb{D}}(1-|\varphi_a(z)|^2)|f'(\varphi_a(z))|\nonumber\\&=&\|f\|_{\mathcal{B}},\label{1.3}\end{eqnarray} for all $a\in \mathbb{D}$, and the above identities also hold  for all $S\in \mathcal{M}.$  That is, \begin{eqnarray*}\;\mbox{if}\; f\in \mathcal{B}\;\;\mbox{then}\;\; f\circ S \in \mathcal{B}\;\;\mbox{for all}\;\;S\in \mathcal{M}.\end{eqnarray*} Hence $\mathcal{B}$ is a \emph{M\"{o}bius-invariant} space.

  The little Bloch space $\mathcal{B}_0$ consists of all $f\in \mathcal{B}$ such that
$$ \lim\limits_{|z|\rightarrow 1}(1-|z|^2)|f'(z)|=0. $$  Replacing "$\sup\limits_{z\in \mathbb{D}}$" by "$\lim\limits_{|z|\rightarrow 1}$" in (\ref{1.3}),  we get that $f\circ \varphi_a \in \mathcal{B}_0$ for every $f\in \mathcal{B}_0$ and $a\in \mathbb{D}.$  Similarly, $\mathcal{B}_0$ is also a\emph{ M\"{o}bius-invariant} space. Both the Bloch space  $\mathcal{B}$ and the little Bloch space $\mathcal{B}_0$ are Banach spaces under the norm $$ \|f\|_{Bloch}=|f(0)|+\|f\|_{\mathcal{B}}.$$

The above \emph{M\"{o}bius-invariant} spaces have the relationship $B_1\subset B_p\subset B_q \subset \mathcal{B}$ for each $1<p<q<\infty$ (see, e.g \cite[Lemma 1.1]{Tj}). Moreover, $B_1$ is a subset of the little Bloch space $\mathcal{B}_0$ (see \cite{Zhu1}) and the Bloch space $\mathcal{B}$ is maximal among all \emph{M\"{o}bius-invariant} Banach spaces of analytic functions on $\mathbb{D}$ (see, e.g. \cite{RT} ). The Besov space $B_2$ is often referred to as the Dirichlet space $\mathcal{D}$, which is a Hilbert space with inner product $$\langle f,\;g\rangle=f(0)\overline{g(0)}+\int_{\mathbb{D}} f'(z) \overline{g'(z)} dA(z)/\pi.$$

The problem of boundedness and compactness of $C_\varphi$ has been studied in many  function spaces, we refer the readers to the books \cite{CM,Tj,Zhu,Z}.  In the recent time, the papers \cite{BS1,BS2,GM,GS1,G1,MW} play important parts  in the theory of the hypercyclicity  of composition operators $ C_\varphi$ acting on analytic function spaces.

In the following, we introduce some definitions in dynamic systems. Let $L(X)$ denote the space of all linear continuous operators on a \emph{separable infinite dimensional Banach space} $X.$ For a positive integer $n$, the $n$-th  iterate of $T\in L(X)$ denoted by $T^{n},$ is the function obtained by composing $T$ with itself $n$ times.

 A continuous linear operator $T\in L(X)$ is called \emph{hypercyclic} provided
 there is some $f\in X$ such that the orbit  $$Orb(T,f)=\{ T^n f: n=0,1,\cdots\}$$  is dense in $X.$  Such a vector $f$ is said to be a \emph{ hypercyclic} vector for $T.$  Therefore, if a Banach space $X$ admits a \emph{hypercyclic} operator,  $X$ must be separable and infinite dimensional.

  Since the polynomials are dense in the little Bloch space $\mathcal{B}_0$ (see, e.g. \cite[Proposition 3.10]{Zhu}) and the polynomials are dense in $B_p\; (1\leq p<\infty)$ (see, e.g. \cite[Proposition 6.2]{Zhu}), thus the little Bloch space $\mathcal{B}_0$  and $B_p\;(1\leq  p <\infty)$ are \emph{separable infinite dimensional Banach spaces}. This is why we investigate the composition operators on $\mathcal{B}_0$ and $B_p\;(1\leq p<\infty)$.  For motivation, examples and background about linear dynamics we refer the reader to the books \cite{BM} by Bayart and Matheron, \cite{GM} by Grosse-Erdmann and Manguillot, and articles by Godefroy and Shapiro \cite{GS2}. The papers \cite{LZ1}-\cite{LZ3} investigate other aspects of the hypercyclic property.

This paper is inspired by the result \cite[Theorem 1.8]{GM}: "\textbf{No linear  fractional composition operator is \emph{hypercyclic} on the  Dirichlet space $\mathcal{D}$}."  We refer the readers to the paper \cite{WC}, which contains the proof. Now  in this paper, we want to characterize the hypercyclicity of composition operator $C_\varphi$ acting on  $\mathcal{B}_0$ and $B_p$ $(1\leq p<\infty)$. We will show that "\textbf{No  linear  fractional composition operator is \emph{hypercyclic} on the little Bloch space $\mathcal{B}_0$, $B_1$ and $B_p\;(2\leq p<\infty)$}."   Since the Besov space $B_2$ is the  Dirichlet space $\mathcal{D}$, hence we generalize  the above result to some extend. The paper is organized as follows: some lemmas are listed in section 2 and  the main results are given in section 3 and section 4.

Throughout the remainder of this paper, $C$ will denote a positive constant, the exact value of which will vary from one appearance to the next.

\section{auxiliary results}
A linear fractional transformation is a mapping of the form
 $$
\varphi \left( z \right) = \frac{{az + b}}{{cz + d}},
$$
where $ad - bc \ne 0$. We will write $LFT(\mathbb{D})$ to refer to the set of all such maps, which are self-maps of the unit disk $\mathbb{D}$. Those maps that take  $\mathbb{D}$ onto itself are precisely the members of  $Aut(\mathbb{D}),$ so that $Aut(\mathbb{D})\subset LFT(\mathbb{D})\subset S(\mathbb{D}).$

We classify those maps according to their fixed point behaviour, see \cite[p. 5]{Sha}:

(a) \emph{Parabolic} members of $LFT(\mathbb{D})$  have their   fixed point on $\partial \mathbb{D}$.

(b)\emph{ Hyperbolic} members of $LFT(\mathbb{D})$ must have an attractive fixed point in $\overline{\mathbb{D}}$, with the other fixed point outside $\mathbb{D}$, and lying on $\partial \mathbb{D}$ if and only if the map is an automorphism of $\mathbb{D}$.

(c)  \emph{Loxodromic} and \emph{elliptic} members of $\varphi \in LFT(\mathbb{D})$ have a fixed point in $\mathbb{D}$ and a fixed point outside $\overline{\mathbb{D}}$. The elliptic ones are precisely the automorphisms in $LFT(\mathbb{D})$ with this fixed point configuration.

The following two lemmas are well-known, so we omit the details.
\begin{lem}\cite[p. 82 (3.5)]{Zhu} For each $f\in \mathcal{B}$, we have
\begin{eqnarray*} |f(z)|\leq \|f\|_{Bloch} \log\frac{2}{1-|z|^2}. \end{eqnarray*}
\end{lem}

\begin{lem} \cite[Theorem 8]{Zhu1} For every $f\in B_p$ with $ 1< p<\infty$, we have
\begin{eqnarray*} |f(z)|\leq C \|f\|_{B_p}\left(\log \frac{2}{1-|z|^2}\right)^{1-1/p}. \end{eqnarray*}
\end{lem}
\begin{rem} From Lemma 2.1 and Lemma 2.2, we obtain that the norm convergence in $\mathcal{B}_0$ (respectively, $B_p\;(1<p<\infty)$) implies pointwise convergence. \end{rem}
\begin{lem} Let $\varphi\in S(\mathbb{D})$  with an interior fixed point on $\mathbb{D}$. Suppose that $C_\varphi$  is bounded on $\mathcal{B}_0$ (respectively, $B_p\;(1< p<\infty)$). Then the operator $C_\varphi$ is not hypercyclic on $\mathcal{B}_0$  (respectively, $B_p\;(1< p<\infty)$). \end{lem}
\begin{proof} We prove   for the little Bloch space $\mathcal{B}_0$. Let $a\in \mathbb{D}$ be the fixed point of $\varphi.$ Suppose  that $f\in\mathcal{B}_0$ is hypercyclic for $C_\varphi$   and  for each $g\in \mathcal{B}_0$ there exists a sequence $\{n_k\}$ such that $C_{\varphi}^{n_k} f$ tends to $g$ in $\mathcal{B}_0,$ that is, $$\|C_\varphi^{n_k} f-g\|_{Bloch}\rightarrow 0 \;\;\mbox{as}\;\;k\rightarrow\infty.$$ By \emph{Remark 2.3} and $f(\varphi^{n_k}(a))=f(a)$ for every $k\in \mathbb{N},$  it follows that
\begin{eqnarray*} g(a)=\lim\limits_{k\rightarrow\infty} (C_{\varphi}^{n_k}f)(a)=\lim\limits_{k\rightarrow\infty} (C_{\varphi^{n_k}} f)(a)=\lim\limits_{k\rightarrow\infty} f(\varphi^{n_k} (a))=f(a), \end{eqnarray*}  that is not the case for every $g\in \mathcal{B}_0$. Thus  the operator $C_\varphi$ is not hypercyclic on $\mathcal{B}_0$. The proof for the Besov spaces $B_p$ is similar, so we omit the details. This ends the proof.  \end{proof}

\begin{rem} In the following, we need only consider $\varphi$ is \emph{parabolic} or \emph{hyperbolic} case. \end{rem}

The following  lemma is a necessary condition  of the hypercyclic composition operator $C_\varphi$ on $\mathcal{B}_0$ and $B_p.$
\begin{lem}  Suppose that $\varphi\in S(\mathbb{D})$ and the bounded composition operator $C_\varphi$ is hypercyclic on $\mathcal{B}_0$ (respectively, $B_p\;(1< p<\infty)), $ then the compositional symbol $\varphi$ is univalent.
\end{lem}
\begin{proof}We   prove  for the space $\mathcal{B}_0.$ Suppose that $\varphi(z_1)=\varphi(z_2)$ for some $z_1,z_2\in \mathbb{D}$ with $z_1\neq z_2.$ We pick $g\in \mathcal{B}_0$ such that $g(z_1)\neq g(z_2),$ and let $f\in \mathcal{B}_0$ be a hypercyclic vector for $C_\varphi.$  By Lemma 2.1, for each $n\in \mathbb{N},$
\begin{eqnarray*} |C_{\varphi}^nf -g|(z)\leq \|C_\varphi^n f-g\|_{Bloch} \log\frac{2}{1-|z|^2}.\end{eqnarray*} So for   $\epsilon:=|g(z_1)-g(z_2)|>0,$  we choose $n\in \mathbb{N}$ be  large enough so that \begin{eqnarray} |C_{\varphi}^n f -g|(z)<\epsilon/4\;\;\mbox{for}\;\;z=z_1,\;z_2.\label{2.1}\end{eqnarray}
On the other hand, since $\varphi(z_1)=\varphi(z_2)$, it follows that  \begin{eqnarray*}
\epsilon&=&|g(z_1)-g(z_2)|\leq |g(z_1)-C_\varphi^n(f)(z_1)|+|C_\varphi^n (f) (z_1)-g(z_2)|
\\&=&|g(z_1)-C_\varphi^n(f)(z_1)|+|f (\varphi^n (z_1))-g(z_2)|\\&=& |g(z_1)-C_\varphi^n(f)(z_1)|+|f (\varphi^n (z_2))-g(z_2)|
\\&=& |g(z_1)-C_\varphi^n(f)(z_1)|+|C_\varphi^n(f)(z_2)-g(z_2)| \\&<&\frac{\epsilon}{4}+\frac{\epsilon}{4}=\frac{\epsilon}{2}, \end{eqnarray*} we get a contraction. So $\varphi$ must be univalent. The proof for the $B_p\;(1< p<\infty)$ is similar.  This completes the proof.\end{proof}
\begin{lem}\cite[Schwarz-Pick Lemma]{MSZ}   For $\varphi\in S(\mathbb{D}),$ we have $$\frac{1-|z|^2}{1-|\varphi(z)|^2}|\varphi'(z)|\leq 1$$ for all $z\in \mathbb{D}.$
\end{lem}
\begin{lem} \cite[Theorem 6.7]{G2} Let $T$ be an operator on a complex Fr\'{e}chet space $X$. If $x\in X$ is such that $\{\lambda T^n x,\;\lambda\in \mathbb{C},\;|\lambda|=1,\;\mbox{and}\;n\in \mathbb{N}_0\}$ is dense in $X$, then $orb(x,\lambda T)$ is dense in $X$ for each $\lambda\in \mathbb{C}$ with $|\lambda|=1.$

In particular, for any $\lambda\in \mathbb{C}$ with $|\lambda|=1,$ $T$ and $\lambda T$ have the same hypercyclic vectors, that is, \begin{eqnarray}HC(T)=HC(\lambda T).\label{2.2}\end{eqnarray} \end{lem}

\section{Hypercyclicity on the little Bloch space $\mathcal{B}_0$ }

For $\varphi\in S(\mathbb{D})$ and by the Schwarz-Pick lemma (Lemma 2.7), we obtain that
\begin{eqnarray*}\|C_\varphi f\|_{Bloch}&=&|f(\varphi(0))|+\sup\limits_{z\in \mathbb{D}} (1-|z|^2)|f'(\varphi(z))\varphi'(z)|\nonumber\\&\leq & |f(\varphi(0))|+\sup\limits_{z\in \mathbb{D}} (1-|\varphi(z)|^2)|f'(\varphi(z))|\\&\leq& |f(\varphi(0))|+\|f\|_{\mathcal{B}}<\infty,
 \end{eqnarray*}
then the composition operator $C_\varphi$  is always a bounded operator from  $\mathcal{B}$ into $\mathcal{B}$. Moreover, if $\varphi\in \mathcal{B}_0,$ then $C_\varphi$ maps $\mathcal{B}_0$ into $\mathcal{B}_0.$ In this section, we always assume $C_\varphi$ is bounded on $\mathcal{B}_0.$\vspace{1mm}

 $\textbf{Case\;I.} $  $\varphi\in Aut(\mathbb{D}).$
\begin{thm}Suppose $\varphi\in  Aut(\mathbb{D})$ and the composition operator $C_\varphi$ is bounded on $\mathcal{B}_0.$ Then  $C_\varphi$ is not hypercyclic on $\mathcal{B}_0.$   \end{thm}

\begin{proof}
Suppose  that there is an $f\in \mathcal{B}_0$ such that the set $\{C_{\varphi}^k f:\;k\in \mathbb{N}\cup\{0\}\}$ is dense in $\mathcal{B}_0$. By (\ref{1.3}) it follows that
\begin{eqnarray*} \|C_\varphi^k f\|_{Bloch}&=&|f(\varphi^k(0))|+\|f\circ \varphi^k\|_{\mathcal{B}}
\nonumber\\&=& |f(\varphi^k(0))|+\|f\|_{\mathcal{B}}. \end{eqnarray*}
For $f_1(z)=z\in \mathcal{B}_0,$ there exists a subsequence $\{\varphi^{k_j}\}_j$ such that
\begin{eqnarray} \|C_{\varphi}^{k_j} f-f_1\|_{Bloch}\rightarrow 0\;\;\mbox{as}\;\;j\rightarrow\infty.\label{3.1} \end{eqnarray}
From which and \emph{Remark} 2.3, it is clear that for $z=0,$
\begin{eqnarray} f(\varphi^{k_j}(0))\rightarrow 0\;\;\mbox{as}\;\; j\rightarrow \infty.\label{3.2}\end{eqnarray}
By (\ref{3.1}) and (\ref{3.2}), we have
\begin{eqnarray}\|f\|_{Bloch}&=&|f(0)|+\|f\|_{\mathcal{B}}\nonumber\\&=&|f(0)|+\|C_\varphi^{k_j} f\|_{Bloch}-|f(\varphi^{k_j}(0))|\nonumber\\&\rightarrow& |f(0)|+\|f_1\|_{Bloch},\;\;j\rightarrow\infty \nonumber\\&=&|f(0)|+1.\label{3.3} \end{eqnarray}
On the other hand,  there exists another sequence $\{\varphi^{\bar{k}_j}\}_j$ such that
\begin{eqnarray*} \|C_{\varphi}^{\bar{k}_j} f -f_1^2\|_{Bloch} \rightarrow 0\;\;\mbox{as}\;\;j\rightarrow \infty.\end{eqnarray*}
Similarly, $f\circ \varphi^{\bar{k}_j}(0)\rightarrow 0\;\;\mbox{as}\;\;j\rightarrow \infty.$  Besides
\begin{eqnarray} \|f\|_{Bloch}&=&|f(0)| +\|f\|_{\mathcal{B}}\nonumber\\&=& |f(0)| +\|C_{\varphi}^{\bar{k}_j} f\|_{Bloch}-|f\circ \varphi^{\bar{k}_j}(0)|\nonumber\\&\rightarrow&|f(0)| +\|f_1^2\|_{Bloch},\;\;j\rightarrow\infty\nonumber\\&=&|f(0)|+\frac{2\sqrt{3}}{9}.\label{3.4} \end{eqnarray}
Combining (\ref{3.3}) and (\ref{3.4}), we get a contraction. Therefore, the composition operator $C_\varphi$ is not hypercyclic on $\mathcal{B}_0.$ This completes the proof.
\end{proof}

$\textbf{Case\; II.}$ $\varphi \notin Aut (\mathbb{D}).$ For this case, we only consider  $\varphi$ with no interior fixed point in $\mathbb{D}$ by \emph{Remark} 2.5.
\begin{thm} Suppose $\varphi\notin Aut(\mathbb{D})$ and $\varphi$ has no interior fixed point in $\mathbb{D}$. Further assume that the composition operator $C_\varphi$ is bounded on $\mathcal{B}_0$, then   $C_\varphi$ is still not hypercyclic on $\mathcal{B}_0.$
\end{thm}
\begin{proof}
Suppose that $C_\varphi$ is hypercyclic on $\mathcal{B}_0$ and $f\in \mathcal{B}_0$  is a hypercyclic vector for $€C_\varphi.$ Hence for every $g\in \mathcal{B}_0$, there exists   $\{\varphi^{k}\}_k$ satisfying
\begin{eqnarray*} \|C_{\varphi^{k}}f -g\|_{Bloch}\rightarrow 0\;\;\mbox{as}\;\;k\rightarrow \infty. \end{eqnarray*}
In particular, we choose $g(z)=nz$ for a fixed $n\in \mathbb{N},$  then there exists a subsequence $\{\varphi^{k_j}\}_j$ such that
\begin{eqnarray*} \|C_{\varphi^{k_j}} f-nz\|_{Bloch}\rightarrow0 \;\;\mbox{as}\;\;j\rightarrow \infty. \end{eqnarray*}
From which and \emph{Remark} 2.3, it follows that \begin{eqnarray}f(\varphi^{k_j}(0))\rightarrow 0\;\;\mbox{as}\;\;j\rightarrow\infty.\label{3.5}\end{eqnarray}
By Schwarz-Pick Lemma (Lemma 2.7) and (\ref{3.5}), we get
\begin{eqnarray}&&\|C_{\varphi^{k_j}} f\|_{Bloch}=|f(\varphi^{k_j}(0))|+\sup\limits_{z\in \mathbb{D}}(1-|z|^2)|(f\circ \varphi^{k_j})'(z)|\nonumber
\\&&= |f(\varphi^{k_j}(0))|+\sup\limits_{z\in \mathbb{D}}(1-|z|^2)|f'(\varphi^{k_j}(z)) \left(\varphi^{k_j}\right)'(z)|\nonumber\\&&=
|f(\varphi^{k_j}(0))|+\sup\limits_{z\in \mathbb{D}}\frac{1-|z|^2}{1-|\varphi^{k_j}(z)|^2}|\left(\varphi^{k_j}\right)'(z)| (1-|\varphi^{k_j}(z)|^2)|f'(\varphi^{k_j}(z))|\nonumber\\&&\leq|f(\varphi^{k_j}(0))| +\sup\limits_{z\in \mathbb{D}} (1-|\varphi^{k_j}(z)|^2)|f'(\varphi^{k_j}(z))|\nonumber\\&&\leq |f(\varphi^{k_j}(0))| +\|f\|_{\mathcal{B}}\nonumber\\&&\rightarrow \|f\|_{\mathcal{B}}=\|f\|_{Bloch}-|f(0)|<\infty,\;\;j\rightarrow \infty. \label{3.6} \end{eqnarray}
At the same time, since \begin{eqnarray*}\|nz\|_{\mathcal{B}}=\sup\limits_{z\in \mathbb{D}} (1-|z|^2)n=n\rightarrow \infty\;\;\mbox{as}\;\;n\rightarrow\infty.\end{eqnarray*}
Thus \begin{eqnarray} \|C_{\varphi ^{k_j}}f\|_{Bloch}\rightarrow\infty,\;\;j\rightarrow \infty.\label{3.7}\end{eqnarray}
Combining (\ref{3.6}) and (\ref{3.7}), we get a contraction. Thus  $C_\varphi$ is not hypercyclic on $\mathcal{B}_0.$ This completes the proof. \end{proof}
From Lemma 2.8, we obtain that $HC (C_\varphi)=HC (\lambda C_\varphi)$ for $|\lambda|=1,$ hence the following corollary holds.
\begin{cor} For every $\lambda\in \mathbb{C}$ with $|\lambda|=1$ and $\varphi\in S(\mathbb{D})$. Suppose that the composition operator $C_\varphi$ is bounded in $\mathcal{B}_0,$ then the operator  $\lambda C_\varphi$ is not hypercyclic on $\mathcal{B}_0$. \end{cor}
Since the Besov space $B_1\subset \mathcal{B}_0$, we obtain the following corollary,
\begin{cor} For every $\lambda\in \mathbb{C}$ with $|\lambda|=1$ and $\varphi\in S(\mathbb{D})$. Suppose that the composition operator $C_\varphi$ is bounded in $B_1,$ then the operator  $\lambda C_\varphi$ is not hypercyclic on $B_1$. \end{cor}

\section{Hypercyclicity on  $B_p\;$}
This section is similar to section 3, we include the brief proof for the convenience of the readers.\vspace{1mm}\\
$\textbf{Case\;I}.$  $\varphi\in Aut(\mathbb{D}).$
\begin{thm}Suppose $\varphi\in  Aut(\mathbb{D})$ and  the composition operator  $C_\varphi$ is bounded on $B_p$ ($1< p<\infty$), then $C_\varphi$ is not hypercyclic on $B_p.$   \end{thm}

\begin{proof}
Suppose that there is an $f\in B_p$ such that $\{C_{\varphi}^k f:\;k\in \mathbb{N}\cup\{0\}\}$ is dense in $B_p$. By (\ref{1.2}) it follows that
\begin{eqnarray*} \|C_\varphi^k f\|_{p}^p&=&|f(\varphi^k(0))|^p+\|f\circ \varphi^k\|_{B_p}^p
\nonumber\\&=& |f(\varphi^k(0))|^p+\|f\|_{B_p}^p. \end{eqnarray*}
For $f_1(z)=z\in B_p,$ there exists a subsequence $\{\varphi^{k_j}\}$ such that
\begin{eqnarray*} \|C_{\varphi}^{k_j} f-f_1\|_{p}\rightarrow 0\;\;\mbox{as}\;\;j\rightarrow\infty. \end{eqnarray*}
From which and  \emph{Remark} 2.3, it is clear that for $z=0,$
\begin{eqnarray*} f(\varphi^{k_j}(0))\rightarrow 0\;\;\mbox{as}\;\; j\rightarrow \infty.\end{eqnarray*}
By the above three formulas, we have
\begin{eqnarray}\|f\|_{p}^p&=&|f(0)|^p+\|f\|_{B_p}^p\nonumber\\&=&|f(0)|^p+\|C_\varphi^{k_j} f\|_{p}^p-|f(\varphi^{k_j}(0))|^p\nonumber\\&\rightarrow& |f(0)|^p+\|f_1\|_{p}^p,\;\;j\rightarrow\infty \nonumber\\&=&
                                                 |f(0)|^p+\frac{1}{2(p-1)}.\label{4.1}
 \end{eqnarray}
On the other hand,  there exists another sequence $\{\varphi^{\hat{k}_j}\}$ such that
\begin{eqnarray*} \|C_{\varphi}^{\hat{k}_j} f -f_1^2\|_{p} \rightarrow 0\;\;\mbox{as}\;\;j\rightarrow \infty.\end{eqnarray*}
Similarly, $f\circ \varphi^{\hat{k}_j}(0)\rightarrow 0\;\;\mbox{as}\;\;j\rightarrow \infty.$ Using the $Beta$ function $$B(p,q)=\int_0^1 x^{p-1}(1-x)^{q-1} dx,\;\;\mbox{for}\;\;p>0,\;q>0,$$ it follows that
\begin{eqnarray} \|f\|_{p}^p&=&|f(0)|^p +\|f\|_{B_p}^p\nonumber\\&=& |f(0)|^p +\|C_\varphi^{\hat{k}_j}f\|_{p}^p-|f\circ \varphi^{\hat{k}_j}(0)|\nonumber\\&\rightarrow&|f(0)|^p +\|f_1^2\|_{p}^p,\;\;j\rightarrow \infty\nonumber\\&=&                                                                                            |f(0)|^p+2^{p-1}B(\frac{p}{2}+1,p-1).\label{4.2} \end{eqnarray}
From (\ref{4.1}) and (\ref{4.2}), we get a contraction. Thus the composition operator $C_\varphi$ is not hypercyclic on $B_p \;(1< p<\infty).$ This completes the proof.
\end{proof}

$\textbf{Case\; II}.$ $\varphi \notin Aut (\mathbb{D}),$ we only consider the case $p-2\geq0.$
\begin{thm} Suppose $\varphi\notin Aut(\mathbb{D})$ and $\varphi$ has no interior fixed point in $\mathbb{D}$. Further assume that  the composition operator $C_\varphi$ is bounded on $B_p\;(2\leq p<\infty)$, then $C_\varphi$ is still not hypercyclic on $B_p.$
\end{thm}
\begin{proof}
Suppose that $C_\varphi$ is hypercyclic on $B_p$ and  $f\in B_p$ is a hypercyclic vector for $C_\varphi$. Then for each $g\in B_p$, there exists   $\{\varphi^{k}\}_k$ satisfying
\begin{eqnarray*} \|C_{\varphi^{k}}f -g\|_{p}\rightarrow 0\;\;\mbox{as}\;\;k\rightarrow \infty. \end{eqnarray*}
In particular, we choose $g(z)=nz^2$ for a fixed $n\in \mathbb{N},$  then there exists a subsequence $\{\varphi^{k_j}\}_j$ such that
\begin{eqnarray*} \|C_{\varphi^{k_j}} f-nz^2\|_{p}\rightarrow 0 \;\;\mbox{as}\;\;j\rightarrow \infty. \end{eqnarray*}
From which and \emph{Remark} 2.3, it follows that \begin{eqnarray}f(\varphi^{k_j}(0))\rightarrow 0\;\;\mbox{as}\;\;j\rightarrow\infty.\label{4.3}\end{eqnarray}
For $p\geq 2$, by Schwarz-Pick lemma (Lemma 2.7) and  Lemma 2.6,  we get that
\begin{eqnarray} &&\|C_{\varphi^{k_j}} f\|_{p}^p=|f(\varphi^{k_j}(0))|^p+ \|f\circ \varphi^{k_j}\|_{B_p}^p  \nonumber
\\&&= |f(\varphi^{k_j}(0))|^p+\int_{\mathbb{D}}|(f\circ \varphi^{k_j})' (z) |^p(1-|z|^2)^{p-2}dA(z) \nonumber\\&&=  |f(\varphi^{k_j}(0))|^p+\int_{\mathbb{D}}|f'(\varphi^{k_j} (z)) (\varphi^{k_j})'(z) |^p (1-|z|^2)^{p-2} dA(z)
\nonumber\\&& =  |f(\varphi^{k_j}(0))|^p+\int_{\mathbb{D}}\frac{(1-|\varphi^{k_j}(z)|^2)^p|f'(\varphi^{k_j} (z))|^p}{(1-|\varphi^{k_j}(z)|^2)^p}|(\varphi^{k_j})'(z) |^p (1-|z|^2)^{p-2} dA(z)
\nonumber\\&&=  |f(\varphi^{k_j}(0))|^p+\int_{\mathbb{D}}  \left(\frac{(1-|z|^2)|(\varphi^{k_j})'(z)|}{ 1-|\varphi^{k_j}(z)|^2 }\right)^{p-2}\frac{ (1-|\varphi^{k_j}(z)|^2)^p|f'(\varphi^{k_j} (z))|^p|(\varphi^{k_j})'(z)|^2}{( 1-|\varphi^{k_j}(z)|^2)^{2}} dA(z)
\nonumber\\&&\leq   |f(\varphi^{k_j}(0))|^p+\int_{\mathbb{D}}\frac{ (1-|\varphi^{k_j}(z)|^2)^p|f'(\varphi^{k_j} (z))|^p|(\varphi^{k_j})'(z)|^2}{( 1-|\varphi^{k_j}(z)|^2)^{2}} dA(z)\;\;(\varphi \;\mbox{is univalent})\nonumber\\&& =    |f(\varphi^{k_j}(0))|^p+\int_{\varphi^{k_j}(\mathbb{D})} |f'(w)|^p (1-|w|^2)^{p-2} dA(w)
\nonumber\\&&\leq   |f(\varphi^{k_j}(0))|^p+\int_{ \mathbb{D}} |f'(w)|^p (1-|w|^2)^{p-2} dA(w)
 \nonumber\\&&=   |f(\varphi^{k_j}(0))|^p+ \|f\|_{B_p}^p\nonumber\\&&\rightarrow  \|f\|_{B_p}^p=\|f\|_p^p-|f(0)|^p<\infty,\;\;j\rightarrow \infty. \label{4.4} \end{eqnarray}
At the same time, since \begin{eqnarray*}\|nz^2\|_{p}^p= n^p 2^{p-1}B(\frac{p}{2}+1,p-1)\rightarrow \infty\;\;\mbox{as}\;\;n\rightarrow\infty.\end{eqnarray*}
Thus \begin{eqnarray} \|C_{\varphi^{k_j}} f\|_{p}^p\rightarrow\infty\;\;\mbox{as}\;\;j\rightarrow \infty.\label{4.5}\end{eqnarray}
From the above,  comparing  (\ref{4.4}) with (\ref{4.5}),
 we get a contraction. Thus  $C_\varphi$ is not hypercyclic on $B_p \;(2\leq p<\infty).$ This completes the proof. \end{proof}
\begin{cor} For any $\lambda\in \mathbb{C}$ with $|\lambda|=1$ and $\varphi\in S(\mathbb{D})$. Suppose that the composition operator $C_\varphi$ is bounded in $B_p\;(2\leq p<\infty),$ then the operator  $\lambda C_\varphi$ is not hypercyclic on $B_p$. \end{cor}

\emph{Open question:} Is the composition operator $C_\varphi$ hypercyclic on the space $B_p\;(1<p<2)$ with  $\varphi\notin Aut(\mathbb{D})$?

\end{document}